\documentclass{amsart}

\usepackage{amsfonts}
\usepackage{amsmath}
\usepackage{amsthm}
\usepackage{amssymb}
\usepackage{mathrsfs}
\usepackage{ifsym}
\usepackage{dirtytalk}

\theoremstyle{definition}

\theoremstyle{remark}

\numberwithin{equation}{section}

\begin{document}

\title[]{Left absorption in products of countable orders}
\author{Garrett Ervin} \author{Ethan Gu}

\begin{abstract}
We classify the countable linear orders $X$ for which there is an order $A$ with at least two points such that the lexicographic product $AX$ is isomorphic to $X$. Given such an $X$, we determine every corresponding order $A$, and identify when $X$ is isomorphic to its square. More generally, we characterize the countable orders that embed at least two disjoint convex copies of themselves.
\end{abstract}

\maketitle

\section{Introduction}

A linear order $X$ is \emph{left-absorbing} if there is a linear order $A$ containing at least two points such that the lexicographically ordered cartesian product $AX$ is order-isomorphic to $X$. Left absorption can be thought of as a self-similarity property: if $AX \cong X$, then $X$ can be partitioned into $A$-many intervals each of which is isomorphic to the entire order. Our goal in this paper is to classify the countable left-absorbing linear orders $X$, and moreover determine for each such $X$ the complete list of orders $A$ that $X$ absorbs. 

A linear order $X$ is \emph{right-cancelling} if for every pair of orders $K$ and $L$ such that $KX \cong LX$ we have $K \cong L$. Left-absorbing orders $X$ are not right-cancelling, since for some $A \not\cong 1$ we have $AX \cong X \cong 1X$, where $1$ denotes the order with a single point. It turns out that being absorbing on the left is the only barrier to being cancellable on the right for linear orders: in \cite{Morel}, Morel proved that if $X$ is non-right-cancelling, then $X$ is left-absorbing.\footnote{Morel ordered products anti-lexicographically, so that \emph{left-absorbing} and \emph{right-cancelling} in this paper translate to \emph{right-absorbing} and \emph{left-cancelling} in hers.} Thus our classification of the countable left-absorbing orders may be viewed as a characterization of the countable orders $X$ that cannot in general be cancelled in isomorphisms of the form $KX \cong LX$.

Given a linear order $X$, an interval $I \subseteq X$ is a \emph{convex copy} of $X$ if $I$ is isomorphic to $X$. An order is left-absorbing if and only if it can be covered by a collection of disjoint convex copies of itself. We will consider more generally linear orders that simply contain disjoint convex copies of themselves. Say that an order $X$ is \emph{self-similar} if it contains two convex copies of itself $I$ and $J$ with $I \cap J = \emptyset$. Our approach to characterizing the countable left-absorbing orders will be to first characterize the countable self-similar orders, and then determine which of these orders are actually absorbing. 

Here is the characterization. Identify each natural number $n \in \mathbb{N}$ with the set of its predecessors $\{0, 1, \ldots, n-1\}$.  Let $N$ denote either a fixed natural number $n \geq 1$, or $\mathbb{N}$. Fix a partition of the rationals $\mathbb{Q} = \bigcup_{k \in N} Q_k$ into subsets $Q_k$, each of which is dense in $\mathbb{Q}$. For each $k \in N$, fix a countable linear order $I_k$. Let $\mathbb{Q}[I_k]$ denote the order obtained by replacing each point $q \in \mathbb{Q}$ by an order from $\{I_k\}_{k \in N}$, so that $q$ is replaced by $I_k$ if $q \in Q_k$. Following \cite[pg. 116]{Rosenstein}, we call $\mathbb{Q}[I_k]$ the \emph{shuffle} of the orders $I_k$. An interval $I \subseteq \mathbb{Q}[I_k]$ is \emph{negligible} if $I$ is a subinterval of one of the $I_k$, otherwise $I$ is \emph{non-negligible}.  

\theoremstyle{definition}
\newtheorem*{introthm}{Theorem}
\begin{introthm}\label{introthm} \,\ 
Suppose that $X$ is a countable linear order. Then $X$ is self-similar if and only if there is a shuffle $\mathbb{Q}[I_k]$ and a non-negligible interval $I \subseteq \mathbb{Q}[I_k]$ such that $X$ is isomorphic to $I$.  
\end{introthm}

After proving the theorem, we will show that whether a countable self-similar order $X$ is left-absorbing depends on the location in $\mathbb{Q}[I_k]$ of the cuts made by the left and righthand sides of the interval $I$. In the cases when $X$ is left-absorbing, it is where these cuts fall that moreover determines the orders $A$ that are absorbed by $X$. We give the precise statement below in Theorem \ref{thm2}. Typically there are many such $A$. In particular, we will show that $X$ absorbs \emph{every} countable order if and only if $X$ is isomorphic to a shuffle $\mathbb{Q}[I_k]$. 

We were originally motivated by the problem of characterizing the countable orders $X$ that are isomorphic to their lexicographic squares $X^2$. When $X$ has no endpoints, or only a single endpoint, such orders are characterized in \cite{Ervin}. The case when $X$ has two endpoints turns out to be harder, and inspired the investigation of self-similar orders that led to our more general characterization. After we prove it, we will use it to characterize the countable orders with both endpoints that are isomorphic to their squares. 

Apart from the problems section at the end of the paper, we restrict our attention to countable orders. Though our approach can be used to get information about self-similar orders in general, that we are able to characterize such orders when they are countable relies on the fact that there is a unique (up to endpoints) countable dense order type, that of the rationals $\mathbb{Q}$. The order types of uncountable dense linear orders are very diverse, and partly as a result it seems that any classification of the uncountable self-similar orders, even of a given cardinality, would be significantly more complicated. 

\section{Preliminaries}

All linear orders are ordered strictly. We will refer to linear orders $(X, <)$ by their underlying sets $X$. The cardinality of a linear order is the cardinality of its underlying set. \emph{Countable} means finite or countably infinite. \emph{Order} always means linear order, and \emph{isomorphic} always means order-isomorphic. We write $X \cong Y$ to mean $X$ is isomorphic to $Y$. We will not distinguish notationally between the order relations $<_X$ and $<_Y$ of different orders $X$ and $Y$, but use $<$ for all order relations. 

An \emph{order type} is an isomorphism class of linear orders. Two orders have the same type if and only if they are isomorphic. We may occasionally confuse an order with its order type.

If $X$ is a linear order, then $X^*$ denotes the reverse order. That is, $X$ and $X^*$ share the same underlying set of points, but we have $x < y$ in $X^*$ if and only if $x > y$ in $X$. 

Given an order $X$, a subset $I \subseteq X$ is an \emph{interval} if it is convex, that is, whenever $x < z < y$ and $x, y \in I$ then $z \in I$. Given $x, y \in X$ with $x \leq y$, we use $(x, y)$ to denote the interval $\{z \in X: x < z < y\}$. Likewise $[x, y)$, $(x, y]$, and $[x, y]$ have their usual meanings. It need not be true that every interval $I \subseteq X$ can be written in one of these four forms; this holds if and only if every non-empty subset of $X$ has both a greatest lower bound and least upper bound. Singletons $\{x\} = [x, x]$ are considered intervals. Given an arbitrary pair of points $x, y \in X$, we define $[\{x, y\}]$ to be $[x, y]$ if $x \leq y$, or $[y, x]$ if $y < x$. 

An interval $I \subseteq X$ is \emph{open} if it contains neither a left nor right endpoint, \emph{half-open} if it has an endpoint on one side but not the other, and \emph{closed} if it has endpoints on both sides. Intervals of the form $(x, y)$ need not be open, but intervals of the form $[x, y]$ are always closed.

If $I, J \subseteq X$ are intervals, we write $I < J$ if $I$ lies entirely to the left of $J$, that is, if for every $x \in I$ and $y \in J$ we have $x < y$. 

An interval $I \subseteq X$ is an \emph{initial segment} of $X$ if whenever $x \in I$ and $y < x$ then $y \in I$. A \emph{final segment} is an interval $J \subseteq X$ whose complement is an initial segment. If $I$ is an initial segment and $J = X \setminus I$ is the corresponding final segment, the pair $(I, J)$ is called a \emph{cut}. We think of a cut $(I, J)$ as the place between the segments $I$ and $J$. Note that to specify a cut, it is enough to specify one of $I$ or $J$. 

If $I \subseteq X$ is any interval, the \emph{left end} of $I$ is the cut $(I_<, I_{\geq})$, where $I_<$ denotes the initial segment $\{x \in X: \forall y \in I (x < y)\}$. Symmetrically, the \emph{right end} if $I$ is the cut $(I_{\leq}, I_>)$, where $I_>$ denotes the final segment of $X$ consisting of points lying strictly above every point in $I$. 

An order $X$ is \emph{dense} if for any two distinct points of $X$ there is a third point lying strictly between them, and moreover $X$ contains at least two points. Given a pair of points $x, y \in X$, if $x < y$ and there is no $z$ such that $x < z < y$, we say that $y$ is the \emph{successor} of $x$ and $x$ is the \emph{predecessor} of $y$. Thus $X$ is dense if and only if it has at least two points and none of its points has a successor. A subset $Y \subseteq X$ is \emph{dense in $X$} if for any two points of $X$, either both belong to $Y$ or there is a point strictly between them that belongs to $Y$. It is possible for a suborder $Y \subseteq X$ to be dense as a linear order, but not dense in $X$. But if $X$ is dense as a linear order and $Y$ is dense in $X$, then $Y$ must be dense as a linear order.  

We write $\mathbb{Q}$ for the set of rationals equipped with its usual order. Likewise we write $\mathbb{N}$ and $\mathbb{Z}$ for the sets of natural numbers and integers in their usual orders. $\mathbb{N}$ includes $0$. We identify each natural number $n \in \mathbb{N}$ with the set of its predecessors in their usual order $0 < 1 < \ldots < n-1$. 

The order types of countable dense linear orders were characterized by Cantor, who showed that if $X$ is a countable dense linear order with neither a left nor right endpoint, then $X$ is isomorphic to $\mathbb{Q}$. It follows that, up to isomorphism, there are exactly four countable dense orders, which we denote $\mathbb{Q}$, $1 + \mathbb{Q}$, $\mathbb{Q} + 1$, and $1 + \mathbb{Q} + 1.$  The latter three are, respectively, the rationals appended with a lefthand endpoint, with a righthand endpoint, and with both endpoints. Concretely, these orders are isomorphic to $\mathbb{Q} \cap [0, 1)$, $\mathbb{Q} \cap (0, 1]$ and $\mathbb{Q} \cap [0, 1]$.

We will need a generalization of Cantor's theorem, due to Skolem. Suppose $X$ and $Y$ are countable dense linear orders without endpoints (so that both are isomorphic to $\mathbb{Q}$). Let $N$ denote either a fixed natural number $n$, or $\mathbb{N}$. Fix partitions $X = \bigcup_{k \in N} X_k$ and $Y = \bigcup_{k \in N} Y_k$ such that each subset $X_k$ is dense in $X$ and each $Y_k$ is dense in $Y$. Skolem proved that there is an isomorphism $f: X \rightarrow Y$ such that $f[X_k] = Y_k$ for all $k \in N$. Informally, this says that if we color two copies of the rationals $X$ and $Y$ with $N$-many colors, so that every color appears densely often in each, then there is an isomorphism $f$ from $X$ to $Y$ that takes each $x \in X$ to some $y$ of the same color. Both Cantor's theorem and Skolem's generalization are proved by essentially the same well-known back-and-forth argument.  

Given a linear order $X$, and for each $x \in X$ an order $I_x$, we let $X(I_x)$ denote the \emph{replacement of $X$} by the orders $I_x$. This is the order obtained by replacing each point $x \in X$ by the corresponding order $I_x$. Formally, $X(I_x)$ is the set of pairs $\{(x, i): x \in X, i \in I_x\}$ ordered lexicographically by the rule $(x, i) < (y, j)$ if either $x < y$ (in $X$), or $x = y$ and $i < j$ (in $I_x = I_y$). If there is an order $Y$ such that for every $x \in X$ we have $I_x = Y$, then we call the replacement $X(I_x)$ the \emph{lexicographic product} of $X$ and $Y$, and denote it $XY$. 

A replacement $X(I_x)$ is sometimes called an \emph{ordered sum} and denoted $\sum_{x \in X} I_x$. We will usually use the replacement notation, but specifically for replacements of $2$, $\mathbb{N}$, $\mathbb{N}^*$, and $\mathbb{Z}$ we will sometimes use summation notation. More explicitly, given two orders $X$ and $Y$, we write $X + Y$ for the order, unique up to isomorphism, with an initial segment isomorphic to $X$ whose corresponding final segment is isomorphic to $Y$. Formally, we view $X + Y$ as the replacement of $2 = \{0, 1\}$ by the orders $I_0 = X$ and $I_1 = Y$. Given a collection of orders $\{X_i\}_{i \in \mathbb{Z}}$, we write $\ldots + X_{-1} + X_0 + X_1 + X_2 + \ldots$ for the replacement $\mathbb{Z}(X_i)$. We also use the expressions $X_0 + X_1 + \ldots$ and $\ldots + X_1 + X_0$ for the replacements $\mathbb{N}(X_i)$ and $\mathbb{N}^*(X_i)$, respectively.

The sum and product are both associative, in that $(X + Y) + Z \cong X + (Y + Z)$ and $(XY)Z \cong X(YZ)$ for all orders $X, Y, Z$. We freely drop parentheses in such expressions. Neither operation is commutative in general. Products on the right distribute over sums, that is $(X + Y)Z \cong XZ + YZ$ for all $X, Y, Z$. More generally, we have the distribution law $X(I_x)Y \cong X(I_xY)$ for any replacement $X(I_x)$ and order $Y$ multiplied on the right. On the other hand, products on the left do not distribute over sums or replacements in general.  

Given a linear order $X$, an equivalence relation $\sim$ on $X$ is called a \emph{condensation} if all of its equivalence classes are intervals. We write $X / \sim$ for the set of equivalence classes of a condensation $\sim$, and write $c(x)$ for the $\sim$-class of a given $x \in X$. The map $c: X \rightarrow X / \sim$ is called the \emph{condensation map}. Since the members of $X / \sim$ are disjoint intervals of $X$, they are naturally linearly ordered by the rule $c(x) < c(y)$ if $c(x) \neq c(y)$ and $x < y$ in $X$. We call this the \emph{induced order} on $X / \sim$. Observe that the condensation map defines an order-homomorphism from $X$ to $X / \sim$ under the induced order. That is, if $x < y$ in $X$ then $c(x) \leq c(y)$ in $X / \sim$. Conversely, given a linear order $L$ and a surjective homomorphism $c: X \rightarrow L$, the relation $\sim$, defined by $x \sim y$ if $c(x) = c(y)$, is a condensation of $X$, and $X / \sim$ is isomorphic to $L$. 

The notion of a condensation is inverse to that of a replacement. If $L(I_l)$ is a replacement of an order $L$, then the relation $\sim$ on $L$, defined by $x \sim y$ if $x, y$ belong to the same replacing order $I_l$, is a condensation of $L(I_l)$ whose equivalence classes are exactly the $I_l$. We clearly have that $L(I_l) / \sim$ is isomorphic to $L$. And if $L = X / \sim$ is the quotient of an order $X$ by a condensation $\sim$, then $X$ is isomorphic to the natural replacement of $L$ by the orders $c(x)$.

A special kind of replacement of $\mathbb{Q}$ called a shuffle will play a central role in our classification of the left-absorbing orders. These are the replacements of $\mathbb{Q}$ in which every replacing order appears densely often. More precisely, let $N$ denote a fixed natural number $n \geq 1$, or $\mathbb{N}$. Partition the rationals as $\mathbb{Q} = \bigcup_{k \in N} Q_k$ so that each subset $Q_k$ is dense in $\mathbb{Q}$. For each $k \in N$, fix a linear order $I_k$. For each $q \in \mathbb{Q}$, let $I_q = I_k$ if $q \in Q_k$. We denote the replacement $\mathbb{Q}(I_q)$ by $\mathbb{Q}[I_k]$, and call it the \emph{shuffle} of the orders $I_k$. Up to isomorphism, there is only one way to shuffle a given set of orders $\{I_k\}_{k \in N}$. For if $\mathbb{Q} = \bigcup_{k \in N} R_k$ is another partition of the rationals into $N$-many sets $R_k$, each dense in $\mathbb{Q}$, then by Skolem's theorem there is an isomorphism $f: \mathbb{Q} \rightarrow \mathbb{Q}$ that such that $f[Q_k] = R_k$ for all $k \in N$. This $f$ determines an isomorphism between the shuffle $\mathbb{Q}[I_k]$ with respect to the partition $\bigcup_k Q_k$ and the shuffle $\mathbb{Q}[I_k]$ with respect to the partition $\bigcup_k R_k$, namely the map defined by the rule $(q, i) \mapsto (f(q), i)$. Thus there is no harm in not specifying the particular partition used to shuffle the $I_k$. We will also write listed expressions like $\mathbb{Q}[A, B, C]$ to denote the shuffle of the orders $A, B, C$, or $\mathbb{Q}[A_0, A_1, \ldots]$ to denote the shuffle of the orders $A_i$, etc.

The discussion in the previous paragraph shows that the order type of a shuffle $\mathbb{Q}[I_k]$ depends only on the collection of replacing orders $\{I_k\}$. However, when we need to refer to a particular instance of a replacing order, say the order $I_{k_0}$ as the order that replaces the point $q_0 \in Q_{k_0}$, we will think of $\mathbb{Q}[I_k]$ as a concrete replacement $\mathbb{Q}(I_q)$ with respect to a fixed partition $\mathbb{Q} = \bigcup_k Q_k$, and refer to this interval as $I_{q_0}$. 

Given a shuffle $\mathbb{Q}[I_k]$ and an interval $I \subseteq \mathbb{Q}[I_k]$, we say that $I$ is \emph{negligible} if there is $q \in \mathbb{Q}$ such that $I \subseteq I_q$. We note that this definition depends on viewing the shuffle $\mathbb{Q}[I_k]$ as a concrete replacement $\mathbb{Q}(I_q)$, with each $I_q$ drawn from the collection $\{I_k\}$. It is possible that for distinct collections of orders $\{I_k\}$ and $\{J_l\}$, the shuffles $\mathbb{Q}[I_k]$ and $\mathbb{Q}[J_l]$ are isomorphic. Moreover, an order type $I$ may appear as a negligible interval in the first representation $\mathbb{Q}[I_k]$ but not in the second. We return to the issue of non-uniqueness in the representation of shuffles in Section 4. We will show that while a given shuffle can always be represented in more than one way, it has a canonical representation. If the notion of negligible is defined with respect to this canonical representation, it becomes a property of the order type of the shuffle as opposed to a specific representation.

\section{Examples}

In this section we give several typical examples of countable left-absorbing linear orders, as well as an example of a countable self-similar order that is not left-absorbing. We will show in Section 4 that every countable self-similar order $X$, and hence every countable left-absorbing order $X$, appears as a non-negligible interval $I$ in some shuffle $\mathbb{Q}[I_k]$. Before considering specific examples, we begin by showing below that conversely every order that appears as a non-negligible interval in a shuffle is self-similar. 

\subsection{Shuffles and their intervals}

Observe that if $X = \mathbb{Q}$, then for any countable linear order $A$ we have $AX \cong X$. This is because, regardless of the density of $A$ and regardless of whether $A$ has any endpoints, $AX$ is countable, dense, and has no endpoints, and hence is isomorphic to $\mathbb{Q}$ by Cantor's theorem.

More generally, suppose that $X = \mathbb{Q}[I_k]$ is a shuffle of some finite or countably infinite collection of orders $\{I_k\}_{k \in N}$. We claim that $AX \cong X$ for any countable order $A$. To see this, fix a countable order $A$. Let $\mathbb{Q} = \bigcup_{k \in N} Q_k$ be the partition of $\mathbb{Q}$ corresponding to the shuffle $\mathbb{Q}[I_k]$. Then $A\mathbb{Q} \cong \mathbb{Q}$, and moreover $A\mathbb{Q}$ naturally inherits a partition into $N$-many dense subsets from the partition of $\mathbb{Q}$. Explicitly, define $R_k = \{(a, q) \in A\mathbb{Q}: q \in Q_k\}$. Then $A\mathbb{Q} = \bigcup_{k \in N} R_k$ is a partition of $A\mathbb{Q}$ into $N$-many subsets, each of which is dense in $A\mathbb{Q}$. We may view $A\mathbb{Q}[I_k]$ either as the product of $A$ with the shuffle $\mathbb{Q}[I_k]$, or as the replacement of $A\mathbb{Q}$ (which is isomorphic to $\mathbb{Q}$) by the orders $I_k$ according to the partition $A\mathbb{Q} = \bigcup_k R_k$. By Skolem's theorem there is an isomorphism $f: \mathbb{Q} \rightarrow A\mathbb{Q}$ such that $f[Q_k] = R_k$ for all $k \in N$. This isomorphism naturally determines an isomorphism of $X = \mathbb{Q}[I_k]$ with $AX = A\mathbb{Q}[I_k]$, namely the map $F: \mathbb{Q}[I_k] \rightarrow A\mathbb{Q}[I_k]$ defined by $F(q, i) = (f(q), i)$. This map is well-defined exactly because $q \in Q_k$ if and only if $f(q) \in R_k$, so that the second coordinates $i$ are always from the same replacing order $I_k$. Since $F$ is well-defined, it follows immediately from its definition that $F$ is an isomorphism of $A\mathbb{Q}[I_k]$ and $\mathbb{Q}[I_k]$, that is, of $AX$ and $X$.

In Section 4 we will show conversely that if $X$ is a countable order that left-absorbs every countable order $A$, then $X$ is a shuffle. 

Fix a shuffle $\mathbb{Q}[I_k]$. It will be important for us to understand the possible forms an interval $I \subseteq \mathbb{Q}[I_k]$ can take. First, observe that it follows from Cantor's theorem that every interval $J \subseteq \mathbb{Q}$ is isomorphic to one of $\mathbb{Q}$, $1 + \mathbb{Q}$, $\mathbb{Q}+1$, or $1 + \mathbb{Q} + 1$, depending on whether $J$ is open, half-open to the right, half-open to the left, or closed, respectively. If $\mathbb{Q}$ comes equipped with a partition into dense subsets $\mathbb{Q} = \bigcup_k Q_k$, then $J = \bigcup_k (Q_k \cap J)$ is a partition of $J$ into dense subsets. If $J$ is open, then by Skolem's theorem there is an isomorphism $f$ between $\mathbb{Q}$ and $J$ that respects the two partitions. It follows that $\mathbb{Q}[I_k] \cong J[I_k]$, as witnessed by the isomorphism $(q, i) \mapsto (f(q), i)$. 

Fix an interval $I \subseteq \mathbb{Q}[I_k]$. Let $c: \mathbb{Q}[I_k] \rightarrow \mathbb{Q}$ be the natural condensation map onto $\mathbb{Q}$ defined by $c(x) = q$ if $x \in I_q$. Then $c[I] = \{q \in \mathbb{Q}: I \cap I_q \neq \emptyset\}$. The interval $I$ is negligible if $I \subseteq I_q$ for some $q \in \mathbb{Q}$, that is, if $c[I]$ is a singleton. If $I$ is not negligible, let $c[I]^-$ denote the interval obtained by deleting any endpoints from $c[I]$, if they exist. Then (since $\mathbb{Q}$ is dense) $c[I]^-$ is an open interval in $\mathbb{Q}$, and hence $c[I]^-[I_k]$ is isomorphic to $\mathbb{Q}[I_k]$. 

If $c[I]$ has a left endpoint $l$, then $I_l$ is the leftmost interval among the $I_q$ for which $I \cap I_l \not\cong \emptyset$. Notice that $I \cap I_l$ is a final segment of $I_l$, namely the final segment $L$ whose left end coincides with the left end of $I$. We have $L = I_l$ if the left ends of $I$ and $I_l$ coincide, that is, if $I_l$ is an initial segment of $I$. Symmetrically, if $c[I]$ has a right endpoint $r$, then $I$ has a final segment $R$ that is an initial segment of $I_r$, and is equal to $I_r$ if the right ends of $I$ and $I_r$ coincide. Combining these observations with those in the previous paragraph, we see that if $I$ is a non-negligible interval of $\mathbb{Q}[I_k]$, we have $I \cong L + \mathbb{Q}[I_k] + R$, where $L$ is the final segment of some $I_{k_0}$, or empty, and $R$ is an initial segment of some $I_{k_1}$, or empty. Conversely, every order of this form can be realized as an interval in $\mathbb{Q}[I_k]$. For example, say, if $L$ is a final segment of $I_{k_0}$ and $R$ is an initial segment of $I_{k_1}$, choose $l, r \in \mathbb{Q}$ with $l < r$ such that $I_l = I_{k_0}$ and $I_r = I_{k_1}$. This is always possible since the points in $\mathbb{Q}$ replaced by $I_{k_0}$ are dense, and likewise for $I_{k_1}$. Let $I \subseteq \mathbb{Q}[I_k]$ be the interval whose left end corresponds with the left end of $L$ in $I_l$, and whose right end corresponds with the right end of $R$ in $I_r$. Then $I \cong L + \mathbb{Q}[I_k] + R$, as desired. The cases when one or both of $L$ and $R$ are $\emptyset$ are similar.

\theoremstyle{definition}
\newtheorem*{prop}{Proposition}
\begin{prop}\label{prop} \,\ 
If $I \subseteq \mathbb{Q}[I_k]$ is a non-negligible interval, then $I$ is self-similar.
\end{prop}

\begin{proof} By above, we have $I \cong L + \mathbb{Q}[I_k] + R$ for $L$ a final segment of some $I_{k_0}$, or empty, and $R$ an initial segment of some $I_{k_1}$, or empty. Assume neither $L$ nor $R$ is empty; if either one is empty, the following construction can be easily modified. Choose $l < r < l' < r'$ in $\mathbb{Q}$ such that $I_l = I_{l'} = I_{k_0}$ and $I_r = I_{r'} = I_{k_1}$. Let $I_0$ be the interval whose condensation $c[I_0]$ is the interval $[l, r]$ in $\mathbb{Q}$, such that $I \cap I_l = L$ and $I \cap I_r = R$. Likewise let $I_1$ be the interval with $c[I_1] = [l', r']$ such that $I_1 \cap I_{l'} = L$ and $I_1 \cap I_{r'} = R$. Then by construction $I \cong I_0 \cong I_1 \cong L + \mathbb{Q}[I_k] \cong R$. Since $I_0 \cap I_1 = \emptyset$, we have that $I$ is self-similar. \end{proof}

We will show in Section 4 that conversely if $X$ is a countable self-similar order, then $X$ is isomorphic to a non-negligible interval $I \cong L + \mathbb{Q}[I_k] + R$ of some shuffle $\mathbb{Q}[I_k]$. Moreover, we will show that such an $X$ is left-absorbing if and only if either $L = R = \emptyset$, or $L = I_{k_0}$ and $R = I_{k_1}$ for some shuffled orders $I_{k_0}$ and $I_{k_1}$, or $R + L = I_k$ for some shuffled order $I_k$.

\subsection{Specific examples} We now give some specific examples of countable self-similar and left-absorbing linear orders. We begin with left products of the four countable dense order types.

If $X = \mathbb{Q}$, then we have $AX \cong X$ for every countable order $A$, as we observed. 

If $X = 1 + \mathbb{Q}$, then if $A$ is countable and has a left endpoint, $AX$ will also be countable and have a left endpoint. Furthermore, regardless of the density of $A$ and regardless of whether $A$ has a right endpoint, $AX$ will be dense and have no right endpoint, so that $AX \cong 1 + \mathbb{Q} \cong X$. If $A$ does not have a left endpoint, then $AX$ will also have no left endpoint but will still be dense, so that $AX \cong \mathbb{Q}$. Hence $AX \cong X$ if and only if $A$ is countable and has a left endpoint. 

Symmetrically, if $X = \mathbb{Q} + 1$, we have $AX \cong X$ if and only if $A$ is countable and has a right endpoint.

The case when $X = 1 + \mathbb{Q} + 1$ is peculiar. For $AX \cong X$ to hold, $A$ must have both a left and right endpoint, since otherwise $AX$ will be missing at least one endpoint. But unlike in the previous cases, $A$ must also be dense. For if there are points $a, b \in A$ with $b$ the successor of $a$, then since in the product $AX$ the interval $I_a \cong 1 + \mathbb{Q} + 1$ has a top point $x$, and $I_b = 1 + \mathbb{Q} + 1$ has a bottom point $y$, there is a point in $AX$ (namely $x$) with a successor (namely $y$), so that $AX$ is not dense and thus not isomorphic to $X = 1 + \mathbb{Q} + 1$. Up to isomorphism there is only one countable dense linear order with both endpoints, namely $1 + \mathbb{Q} + 1$. It is not hard to see that $(1 + \mathbb{Q} + 1)(1 + \mathbb{Q} + 1) \cong 1 + \mathbb{Q} + 1$. Hence in this case $AX \cong X$ if and only if $A \cong X = 1 + \mathbb{Q} + 1$.

Thus each of the orders $\mathbb{Q}$, $1 + \mathbb{Q}$, $\mathbb{Q} + 1$, and $1 + \mathbb{Q} + 1$ is left-absorbing. The class of orders $A$ absorbed by a given one of the four depends on its endpoint configuration. We will see that this dependence reappears in more complicated left-absorbing orders, though in a somewhat different form. In general, which orders $A$ are absorbed by a given countable left-absorbing order $X$ depends on whether $X$ has an initial segment (akin to a left endpoint) that contains no convex copy of itself, and also whether $X$ has a final segment (akin to a right endpoint) that contains no convex copy of itself. 

To illustrate this, here are some more examples. Suppose $X = \mathbb{Q}[\mathbb{Z}] = \mathbb{Q}\mathbb{Z}$. Then $AX = A(\mathbb{Q}\mathbb{Z}) \cong (A\mathbb{Q})\mathbb{Z} \cong \mathbb{Q}\mathbb{Z} = X$ for every countable order $A$. In this case, every non-empty initial and final segment of $X$ contains a convex copy of $X$, as the reader is invited to check. We will see that all countable left-absorbing orders $X$ with this property absorb the same countable orders that $\mathbb{Q}$ does (i.e. all of them).

Now let $X = \mathbb{Z} + \mathbb{Q}\mathbb{Z}$. We claim that $AX \cong X$ if and only if $A$ is countable and has a left endpoint. To see this, fix such an $A$. We have $AX = A(\mathbb{Z} + \mathbb{Q}\mathbb{Z}) \cong A(1+\mathbb{Q})\mathbb{Z} \cong (1+\mathbb{Q})\mathbb{Z} \cong X$, as desired. And if $A$ has no left endpoint, then $AX \cong A(1+\mathbb{Q})\mathbb{Z} \cong \mathbb{Q}\mathbb{Z}$. Since $X = \mathbb{Z}+ \mathbb{Q}\mathbb{Z}$ has an initial segment isomorphic to $\mathbb{Z}$ and $\mathbb{Q}\mathbb{Z}$ does not, we have $AX \not\cong X$ in this second case, which shows the claim. Notice that there is an initial segment of $X$ that does not contain a convex copy of $X$ (for example the initial copy of $\mathbb{Z}$), but every non-empty final segment of $X$ contains a convex copy of $X$. We will see that all countable left-absorbing orders with this property absorb the same orders as $1 + \mathbb{Q}$.

Symmetrically, if $X = \mathbb{Q}\mathbb{Z} + \mathbb{Z}$ we have $AX \cong X$ if and only if $A$ is countable and has a right endpoint. In this case, $X$ has a non-empty final segment containing no convex copy of $X$, but no such initial segment. All such countable left-absorbing orders absorb the same orders as $\mathbb{Q} + 1$. 

If $X = \mathbb{Z} + \mathbb{Q}\mathbb{Z} + \mathbb{Z}$, then similarly we have $AX \cong X$ if and only if $A \cong 1 + \mathbb{Q} + 1$. Observe that $X$ has both an initial and final segment embedding no convex copy of $X$, and that $X$ absorbs the same orders as $1 + \mathbb{Q} + 1$. 

But now there is a difference, since not \emph{all} such $X$ absorb the same orders as $1 + \mathbb{Q} + 1$. Consider $X = \mathbb{N} + \mathbb{Q}\mathbb{Z} + \mathbb{N}^*$. We check that $X$ is left-absorbing. Consider $2X = \mathbb{N} + \mathbb{Q}\mathbb{Z} + \mathbb{N}^* + \mathbb{N} + \mathbb{Q}\mathbb{Z} + \mathbb{N}^* \cong \mathbb{N} + \mathbb{Q}\mathbb{Z} + \mathbb{Z} + \mathbb{Q}\mathbb{Z} + \mathbb{N}^*$. Since $\mathbb{N}^* + \mathbb{N} \cong \mathbb{Z}$, the middle three terms are isomorphic to $(\mathbb{Q} + 1 + \mathbb{Q})\mathbb{Z}$. Since $\mathbb{Q} + 1 + \mathbb{Q}$ is countable, dense, and has no endpoints, we have $\mathbb{Q} + 1 + \mathbb{Q} \cong \mathbb{Q}$. Hence $2X \cong \mathbb{N} + \mathbb{Q}\mathbb{Z} + \mathbb{N}^* = X$. So $X$ absorbs $2$. Moreover $X$ has both an initial and final segment embedding no convex copy of itself, namely the initial $\mathbb{N}$ and final $\mathbb{N}^*$ in the sum $\mathbb{N} + \mathbb{Q}\mathbb{Z} + \mathbb{N}^*$. But the spectrum of orders $A$ absorbed by $X$ is different than for $1 + \mathbb{Q} + 1$. It can be checked that $X$ absorbs $A$ if and only if $A$ is countable, has both a left and right endpoint, and every point in $A$ that is not the left endpoint has a predecessor, and every point that is not the right endpoint has a successor. We will analyze absorption of this kind in Section 4. 

So far, all of our examples have been of left-absorbing orders. Are there countable self-similar orders that are not left-absorbing? Let $X = \mathbb{N} + \mathbb{Q}\mathbb{Z}$. Since $\mathbb{N}$ is a final segment of some (actually every) replacing order $I_q$ in the shuffle $\mathbb{Q}\mathbb{Z} = \mathbb{Q}[\mathbb{Z}]$, we have by the proposition above that $X$ is self-similar. But $X$ is not left-absorbing. For suppose that $A$ is an order with at least two points $a < b$. The initial copy of $\mathbb{N}$ in the sum $\mathbb{N} + \mathbb{Q}\mathbb{Z} = X$ has the property that it is the unique interval $I \subseteq X$ such that $I \cong \mathbb{N}$ and for every interval $I' \supseteq I$ strictly containing $I$ we have $I' \not\cong \mathbb{N}$. But in $AX$ there are at least two such intervals, since the initial copies of $\mathbb{N}$ in the intervals $I_a \cong X$ and $I_b \cong X$ both have this property. Hence $AX \not\cong X$. It follows that $X$ absorbs no non-trivial order $A$ on the left. One might notice that this example differs from the previous one in that the initial copy of $\mathbb{N}$ in $X$ is no longer complemented by a final copy of $\mathbb{N}^*$. This pre-empts placing copies of $X$ next to one another to get a middle copy of $\mathbb{N}^* + \mathbb{N} \cong \mathbb{Z}$, which can then be absorbed into the surrounding copies of $\mathbb{Q}\mathbb{Z}$.

On the other hand, if $X = \mathbb{N} + \mathbb{Q}[\mathbb{Z}, \mathbb{N}]$, then $X$ is left-absorbing. In fact, $AX \cong X$ for every countable order $A$ with a left endpoint. The reason is more or less the same as for $\mathbb{Z} + \mathbb{Q}\mathbb{Z}$, though the proof is slightly more cumbersome to write. We sketch it as follows. For such an $A$, consider the natural condensation of $AX = A(\mathbb{N} + \mathbb{Q}[\mathbb{Z}, \mathbb{N}])$ onto $A(1 + \mathbb{Q}) \cong 1 + \mathbb{Q}$ that condenses every copy of $\mathbb{N}$ and every copy of $\mathbb{Z}$ to a point. Then the set of points that are condensed images of $\mathbb{N}$ and the set of points that are condensed images of $\mathbb{Z}$ are both dense in $1 + \mathbb{Q}$, and the initial point is a condensed image of $\mathbb{N}$. But this also describes the natural condensation of $X = \mathbb{N} + \mathbb{Q}[\mathbb{Z}, \mathbb{N}]$ onto $1 + \mathbb{Q}$. It follows $AX \cong X$. 

\section{Characterizing the self-similar countable orders and left-absorbing countable orders}

We now prove our main results. Theorem \ref{thm1} characterizes the countable self-similar linear orders as exactly the orders that appear as non-negligible intervals in shuffles. The key step in the proof is to condense the intervals of a given self-similar order that are maximal with respect to not containing a convex copy of the order. Theorem \ref{thm2} identifies which countable self-similar orders are left-absorbing, and for each such order determines the orders it absorbs. 

\theoremstyle{definition}
\newtheorem{thm1}{Theorem}
\begin{thm1}\label{thm1} \,\ 
Suppose that $X$ is a countable linear order. Then $X$ is self-similar if and only if there is a shuffle $\mathbb{Q}[I_k]$ and a non-negligible interval $I \subseteq \mathbb{Q}[I_k]$ such that $X \cong I$. 
\end{thm1}

\begin{proof}
Since the backward implication was proved in the previous section, it suffices to prove the forward one. Suppose that $X$ is countable and self-similar. Define a relation $\sim$ on $X$ by the rule $x \sim y$ if the closed interval $[\{x, y\}]$ does not contain a convex copy of $X$.

We claim that $\sim$ is a condensation of $X$. We must show that $\sim$ is an equivalence relation whose equivalence classes are intervals. Reflexivity and symmetry of $\sim$ follow immediately from its definition. To prove transitivity, fix $x, y, z \in X$ and suppose $x \sim y \sim z$. The six possible orderings of the points $x, y, z$ are $x < y < z$, $x < z < y$, $y < x < z$, $y < z < x$, $z < x < y$, and $z < y < x$. For the middle four orderings we get $x \sim z$ immediately. For example, suppose $x < z < y$. Since $x \sim y$, there is no convex copy of $X$ in the interval $[\{x, y\}] = [x, y]$. But then there is no convex copy of $X$ in the smaller interval $[\{x, z\}] = [x, z]$ and we have $x \sim z$. A similar argument applies to the subsequent three orderings. 

Assume $x < y < z$. The argument for the case when $z < y < x$ is symmetric. If $x \not \sim z$, then there is an interval $I \subseteq [x, z]$ that is isomorphic to $X$. Since $X$ is self-similar, there are disjoint subintervals $I_0$ and $I_1$ of $I$ that are each isomorphic to $X$. Suppose without loss of generality that $I_0 < I_1$. We cannot have $I_0 \subseteq [x, y]$ since this contradicts $x \sim y$. But if $I_0 \not \subseteq [x, y]$, then $I_0 \cap [y, z] \neq \emptyset$. Since $I_1$ lies to the right of $I_0$ we must have $I_1 \subseteq [y, z]$, contradicting $y \sim z$. It follows $x \sim z$, as desired. 

It remains to show that the equivalence classes of $\sim$ are convex. But this is clear: if $x < y < z$ and $x \sim z$, then since there is no convex copy of $X$ in $[x, z]$, there is no such copy in the smaller interval $[x, y]$, and we have $x \sim y$. Thus $\sim$ is a condensation of $X$, as claimed. 

It follows from the definition of $\sim$ that for any equivalence class $c(x)$, if $I \subseteq c(x)$ is an interval that is bounded below by some $x_0 \in c(x)$ and above by some $x_1 \in c(x)$, then $I \not\cong X$. On the surface, this leaves open the possibility that an interval $I \subseteq c(x)$ that is unbounded to the right or left in $c(x)$ may be isomorphic to $X$. But this cannot happen: if there were such an $I$, then since $X$ is self-similar we could find subintervals $I_0 < I_1$ of $I$ that are also isomorphic to $X$, and then subintervals $I_{00} < I_{01}$ of $I_0$ that are isomorphic to $X$ as well. But then $I_{01}$ would be a bounded convex copy of $X$ in $c(x)$, which is impossible. Thus each condensation class $c(x)$ contains no convex copy of $X$. 

We next claim that $X / \sim$ is dense as a linear order. Note that by the previous paragraph $X / \sim$ cannot be a singleton, as then the unique condensation class $c(x)$ would be isomorphic to $X$. Suppose toward a contradiction there are classes $c(x), c(y) \in X / \sim$, with representatives $x, y \in X$, such that $c(y)$ is the successor of $c(x)$ in $X / \sim$. Then since $x \not \sim y$, the interval $[x, y]$ contains an interval $I$ that is isomorphic to $X$. Since $X$ is self-similar, $I$ contains convex subcopies $I_0 < I_1$ of $X$. It must be that either $I_0 \subseteq c(x)$ or $I_1 \subseteq c(y)$, a contradiction either way. Thus $X / \sim$ is dense, as claimed. 

Since $X / \sim$ is the condensation of a countable order, it is countable. Since it is dense, it is isomorphic to either $\mathbb{Q}$, $1 + \mathbb{Q}$, $\mathbb{Q} + 1$, or $1 + \mathbb{Q} + 1$. We will show that in each of these cases there is a shuffle $\mathbb{Q}[I_k]$ and orders $L$ and $R$ such that $X \cong L + \mathbb{Q}[I_k] + R$. In the first case we show that $L = R = \emptyset$, in the second, that $L$ is a final segment of some $I_k$ and $R = \emptyset$, in the third, that $L = \emptyset$ and $R$ is an initial segment of some $I_k$, and in the fourth, that $L$ is a final segment of some $I_{k_0}$ and $R$ is an initial segment of some $I_{k_1}$. By the proposition in the previous section, it will follow that in each case there is a non-negligible interval $I \subseteq \mathbb{Q}[I_k]$ such that $X \cong I$, and the theorem will be proved.   

Suppose first that $X / \sim$ is isomorphic to $\mathbb{Q}$. For simplicity, identify $X / \sim$ with $\mathbb{Q}$. For every $q \in \mathbb{Q}$, let $I_q$ denote $c^{-1}(q)$, the interval of points in $X$ that are condensed to $q$. Then $X$ is isomorphic to the replacement $\mathbb{Q}(I_q)$. 

We claim that $X$ is actually isomorphic to a shuffle $\mathbb{Q}[I_k]$. To prove this, it suffices to show that for every $q \in \mathbb{Q}$, the set of $p$ such that $I_p \cong I_q$ is dense in $\mathbb{Q}$. Then, if we identify orders in $\{I_q\}$ that are isomorphic, each $q$ determines a dense subset of $\mathbb{Q}$, namely $Q_q = \{p \in \mathbb{Q}: I_p = I_q\}$. Enumerating these subsets as $Q_k$, and selecting a representative $q_k \in Q_k$ for every $k$, we may take the collection $\{I_k\}$ to be $\{I_{q_k}\}$, and view the shuffle $\mathbb{Q}[I_k]$ as being constructed with respect to the partition $\mathbb{Q} = \bigcup_k Q_k$. 

Fix $q \in \mathbb{Q}$, and also fix $q_0, q_1 \in \mathbb{Q}$ with $q_0 < q_1$. We find $p$ in the interval $(q_0, q_1)$ with $I_p \cong I_q$. Fix $x_0 \in I_{q_0}$ and $x_1 \in I_{q_1}$. Since $x_0 \not\sim x_1$, there is an interval $J \subseteq [x_0, x_1]$ that is isomorphic to $X$. Since there is no leftmost condensation class in $X$, every nonempty initial segment of $X$ contains points from distinct condensation classes. It follows that every initial segment of $X$ contains a convex copy of $X$, and so every initial segment of $J$ also contains a convex copy of $X$. Thus $J \cap I_{q_0} = \emptyset$, since if not, $J \cap I_{q_0}$ would be an initial segment of $J$ containing no convex copy of $X$. Symmetrically, every non-empty final segment of $X$, and therefore every non-empty final segment of $J$, contains a convex copy of $X$, so that $J \cap I_{q_1} = \emptyset$ as well. Thus $J$ lies strictly between $I_{q_0}$ and $I_{q_1}$, that is, $I_{q_0} < J < I_{q_1}$.

Let $f: X \rightarrow J$ be an isomorphism. Consider the image $f[I_q]$ of the interval $I_q$. Choose $y \in f[I_q]$. Then $c(y) = I_p$ for some $p \in \mathbb{Q}$. Since $I_p \cap J \neq \emptyset$, we must actually have $I_p \subseteq J$. Otherwise, as $I_p$ is an interval, it would contain either an initial or final segment of $J$ and hence a convex copy of $X$. It follows then, since $I_{q_0} < J < I_{q_1}$, that we have $q_0 < p < q_1$. We claim that $f[I_q] = I_p$. For the forward containment, notice that if $y' \in f[I_q]$ then we must have $y' \sim y$. Otherwise $f[I_q]$, which contains $[\{y, y'\}]$, would contain a convex copy of $X$. This is impossible, as $f[I_q] \cong I_q$. For the reverse containment, notice that the same argument applies to the image $f^{-1}[I_p]$ under the inverse isomorphism $f^{-1}: J \rightarrow X$, giving $f^{-1}[I_p] \subseteq I_q$. Thus $f[I_q] = I_p$, and so $I_q \cong I_p$. Since $q_0$ and $q_1$ were arbitrary, the set of condensation classes $I_p$ that are isomorphic to $I_q$ is dense in $X / \sim$. Since $q$ was arbitrary, it follows that $X$ is isomorphic to a shuffle $\mathbb{Q}[I_k]$, as claimed. 

Now suppose $X / \sim$ is isomorphic to $1 + \mathbb{Q}$. Identify $X / \sim$ with $1 + \mathbb{Q}$, so that $X$ is isomorphic to a replacement $(1+\mathbb{Q})(I_q)$. Let $l$ denote the left endpoint of $1 + \mathbb{Q}$, so that $X \cong I_l + \mathbb{Q}(I_q)$. Observe that an initial segment $I$ of $X$ contains a convex copy of $X$ if and only if $I$ properly extends $I_l$, but every nonempty final segment of $X$ contains a convex copy of $X$. 

We claim that for some $q \in \mathbb{Q}$, $I_l$ is isomorphic to a final segment $L$ of $I_q$. Fix points $q_0 < q_1$ in $\mathbb{Q}$, and then fix $x_0, x_1 \in X$ from the corresponding condensation classes $I_{q_0}$ and $I_{q_1}$. Let $J \subseteq [x_0, x_1]$ be an interval isomorphic to $X$, and let $f: X \rightarrow J$ be an isomorphism. Since $f[I_l]$ is initial in $J$, there is $q \in [q_0, q_1)$ such that $f[I_l] \cap I_q \neq \emptyset$. We must have $f[I_l] \subseteq I_q$, since otherwise $f[I_l]$ would contain a convex copy of $X$. Consider the interval $L = J \cap I_q$. We have $f[I_l] \subseteq L$. We show the reverse containment. Since $I_q$ is the leftmost condensation class that $J$ intersects, $L$ is a final segment of $I_q$. And since $f^{-1}[L]$ intersects $I_l$, it must be contained in $I_l$. Otherwise it would contain a convex copy of $X$, and then so would $I_q$. Thus $L \subseteq f[I_l]$ and so $f[I_l] = L$, giving $I_l \cong L$ as claimed. 

We have $X \cong L + \mathbb{Q}(I_q)$. We show the righthand term is isomorphic to a shuffle $\mathbb{Q}[I_k]$. As before, it suffices to check that for any $q \in \mathbb{Q}$ there are densely many $p$ such that $I_p \cong I_q$. Fix $q$ in $\mathbb{Q}$, then fix $q_0 < q_1$ in $\mathbb{Q}$, and then $x_0 \in I_{q_0}$ and $x_1 \in I_{q_1}$. We find $p$ such that $q_0 < p < q_1$ and $I_p \cong I_q$. Let $J \subseteq [x_0, x_1]$ be an interval that is isomorphic to $X$, as witnessed by an isomorphism $f: X \rightarrow J$. It may be that $J \cap I_{q_0} \neq \emptyset$. But if this is so, then a similar argument to the one given in the previous paragraph shows that $J \cap I_{q_0} = f[I_l]$. Since $l < q$, we have $I_{q_0} < f[I_q]$. Since $J$ must lie completely to the left of $I_{q_1}$, as otherwise it would have a final segment containing no convex copy of $X$, we have also $f[I_q] < I_{q_1}$. Arguing as in the previous case, we have $f[I_q] \subseteq I_p$ for some unique $p \in (q_0, q_1)$, and then $f^{-1}[I_p] \subseteq I_q$ as well. Thus $I_p = f[I_q]$, and hence $I_p \cong I_q$. Since $q$ was arbitrary, it follows $X \cong L + \mathbb{Q}[I_k]$ as claimed. 

The arguments for the cases when $X / \sim$ is isomorphic to $\mathbb{Q} + 1$ and $1 + \mathbb{Q} + 1$ are similar, and we leave them out. For these, we get respectively that $X \cong \mathbb{Q}[I_k] + R$ and $X \cong L + \mathbb{Q}[I_k] + R$, where $L$ is a final segment of some $I_{k_0}$ and $R$ is an initial segment of some $I_{k_1}$. We are done.
\end{proof}

We turn now to the problem of characterizing the countable left-absorbing linear orders. Before we can state and prove our characterization, we need to address the issue of uniqueness in the representation of shuffles $\mathbb{Q}[I_k]$. 

We will think of two collections of orders $\{I_k\}$ and $\{J_l\}$ as being the same if there is a bijection $i$ between the sets of indices $\{k\}$ and $\{l\}$ such that $I_k \cong J_{i(k)}$ for all $k$. If there is no such a bijection, we think of the collections as being distinct. It can happen that for distinct collections of orders, the corresponding shuffles $\mathbb{Q}[I_k]$ and $\mathbb{Q}[J_l]$ are isomorphic. For example, consider the shuffle $\mathbb{Q}[1, 1 + \mathbb{Q}]$. Here, we have decomposed $\mathbb{Q}$ into two dense subsets $Q_1$ and $Q_2$, and kept the points in $Q_1$ as singletons while substituting every point in $Q_2$ with a copy of $1 + \mathbb{Q}$. The resulting order remains countable and without endpoints. Let us verify, in explicit detail, that it is also dense. View $\mathbb{Q}[1, 1 + \mathbb{Q}]$ as a replacement $\mathbb{Q}(I_q)$, and suppose $x < y$ in $\mathbb{Q}[1, 1 + \mathbb{Q}]$. It may be that $x$ and $y$ belong to the same replacing order $I_q$, in which case we must have $I_q = 1 + \mathbb{Q}$. But then there is a point between $x$ and $y$, since $1 + \mathbb{Q}$ is dense. The other possibility is that $x \in I_q$ and $y \in I_{q'}$ for some $q < q'$ in $\mathbb{Q}$. But then we can also find a point between them, specifically in $I_r$ for some $r \in \mathbb{Q}$ between $q$ and $q'$. Thus $\mathbb{Q}[1, 1 + \mathbb{Q}]$ is dense, so that $\mathbb{Q}[1, 1 + \mathbb{Q}] \cong \mathbb{Q}$. This shows that for the distinct collections $\{1, 1 + \mathbb{Q}\}$ and $\{1\}$, the corresponding shuffles $\mathbb{Q}[1, 1 + \mathbb{Q}]$ and $\mathbb{Q}[1]$ are isomorphic.

Here is another example. Let $X = \mathbb{Q}[\mathbb{N}, \mathbb{Z} + \mathbb{Q}[\mathbb{N}, \mathbb{Z}]]$. Again, we have decomposed $\mathbb{Q}$ into two dense subsets, but now replaced points in the first with $\mathbb{N}$ and points in the second with $\mathbb{Z} + \mathbb{Q}[\mathbb{N}, \mathbb{Z}]$. We claim that $X \cong \mathbb{Q}[\mathbb{N}, \mathbb{Z}]$. Once more we verify this in some detail, since going forward we will only sketch any similar arguments. Let $c: \mathbb{Q}[\mathbb{N}, \mathbb{Z} + \mathbb{Q}[\mathbb{N}, \mathbb{Z}]] \rightarrow \mathbb{Q}[1, 1 + \mathbb{Q}]$ be the natural condensation map of $X$ onto $\mathbb{Q}[1, 1 + \mathbb{Q}]$, in which each copy of $\mathbb{N}$ and each copy of $\mathbb{Z}$ is condensed to a point. We check that the set of points in $\mathbb{Q}[1, 1 + \mathbb{Q}]$ that are condensed images of $\mathbb{Z}$ is dense in $\mathbb{Q}[1, 1 + \mathbb{Q}]$. View the shuffle $\mathbb{Q}[1, 1 + \mathbb{Q}]$ as a replacement $\mathbb{Q}(I_q)$ and fix $x < y$ in $\mathbb{Q}[1, 1 + \mathbb{Q}]$. If $x, y \in I_q$ for some $q$, then it must be $I_q = 1 + \mathbb{Q}$. Since $I_q$ is a condensed copy of $\mathbb{Z} + \mathbb{Q}[\mathbb{N}, \mathbb{Z}]$, there is $z \in I_q$ that is a condensed copy of $\mathbb{Z}$ with $x < z < y$. And if $x \in I_q$ and $y \in I_{q'}$ for some $q < q'$, then there is an $r$ with $q < r < q'$ such that $I_r = 1 + \mathbb{Q}$. Such an $I_r$ contains a point that is a condensed copy of $\mathbb{Z}$. Similarly, the set of points in $\mathbb{Q}[1, 1 + \mathbb{Q}]$ that are condensed images of $\mathbb{N}$ is dense. Label these two dense subsets as $Q_{\mathbb{Z}}$ and $Q_{\mathbb{N}}$ respectively. If we fix an isomorphism  $f: \mathbb{Q}[1, 1 + \mathbb{Q}] \rightarrow \mathbb{Q}$, we also get a partition of $\mathbb{Q}$ into the two dense subsets $R_{\mathbb{Z}} = f[Q_{\mathbb{Z}}]$ and $R_{\mathbb{N}} = f[Q_{\mathbb{N}}]$. If we ``uncondense" $\mathbb{Q}[1, 1 + \mathbb{Q}]$, that is, replace points in $Q_{\mathbb{Z}}$ with $\mathbb{Z}$ and points in $Q_{\mathbb{N}}$ with $\mathbb{N}$, and likewise replace points in $R_{\mathbb{Z}}$ with $\mathbb{Z}$ and points in $R_{\mathbb{N}}$ with $\mathbb{N}$, we can lift $f$ to get an isomorphism $F: \mathbb{Q}[\mathbb{N}, \mathbb{Z} + \mathbb{Q}[\mathbb{N}, \mathbb{Z}]] \rightarrow \mathbb{Q}[\mathbb{N}, \mathbb{Z}]$ by defining $F(q, n) = (f(q), n)$. Thus $\mathbb{Q}[\mathbb{N}, \mathbb{Z} + \mathbb{Q}[\mathbb{N}, \mathbb{Z}]] \cong \mathbb{Q}[\mathbb{N}, \mathbb{Z}]$, as claimed.

Fortunately, representations of shuffles become unique once we insist on a smallness property for the orders being shuffled. Define a countable collection $\{I_k\}$ of countable orders to be \emph{minimal} if none of the $I_k$ contains a convex copy of the shuffle $\mathbb{Q}[I_k]$. We say that a representation $\mathbb{Q}[I_k]$ of a shuffle is minimal if the corresponding collection $\{I_k\}$ is minimal.

Every shuffle has a minimal representation, and this representation is unique. To see this, suppose that $X$ is a shuffle, and consider the condensation $\sim$ employed in the proof of Theorem \ref{thm1}. Since every initial and final segment of $X$ contains a convex copy of $X$, it follows from the proof that $X / \sim$ is isomorphic to $\mathbb{Q}$, so that $X \cong \mathbb{Q}(I_q)$. As in the proof of Theorem \ref{thm1}, if we identify any isomorphic orders in the set of condensation classes $\{I_q\}$, we get a representation $X = \mathbb{Q}[I_k] = \mathbb{Q}[I_{q_k}]$ of $X$ as a shuffle. Since the condensation classes $I_q$ do not contain a convex copy of $X$, this representation is minimal. 

To see the uniqueness of the representation, suppose that $\mathbb{Q}[J_l]$ is another representation of $X$ with respect to some minimal collection $\{J_l\}$. The copies of $J_l$ appearing in $\mathbb{Q}[J_l]$ are the maximal intervals of $\mathbb{Q}[J_l]$ that do not (by minimality) contain a convex copy of $\mathbb{Q}[J_l]$. That is, if $I \subseteq \mathbb{Q}[J_l]$ is an interval that contains no convex copy of $\mathbb{Q}[J_l]$, but every interval $J \supseteq I$ strictly containing $I$ does contain such a copy, then $I \cong J_l$ for some $l$. But such intervals are exactly the condensation classes $I_q$ of the condensation $\sim$ applied to $X = \mathbb{Q}[J_l]$, so that this condensation identifies the orders in $\{J_l\}$ (up to isomorphism). It follows that the representations $\mathbb{Q}[J_l]$ and $\mathbb{Q}[I_k]$ are the same.

Notice that, more generally, the representations of self-similar orders  $X$ as $L + \mathbb{Q}[I_k] + R$ yielded by the proof of Theorem \ref{thm1} are minimal, in the sense that $L$ and $R$ contain no convex copy of $X$, and hence also no convex copy of $\mathbb{Q}[I_k]$. 

Now we can work toward our classification of the countable left-absorbing orders. Fix a minimal collection of countable orders $\{I_k\}$ and two countable orders $L$ and $R$, neither embedding a convex copy of $\mathbb{Q}[I_k]$. Let $X = L + \mathbb{Q}[I_k] + R$. By Theorem \ref{thm1}, such an order is self-similar if and only if $L$ is a final segment of some $I_k$, or empty, and $R$ is an initial segment of some $I_k$, or empty. Our goal is to determine when such an order is left-absorbing, and when it is, determine which orders it absorbs. We will show that such an $X$ is left-absorbing if and only if each of the segments $L$ and $R$ is either empty or isomorphic to one of the $I_k$, or their sum $R + L$ is isomorphic to one of the $I_k$.

We work through the possible cases for $L$ and $R$. To begin, we emphasize that if $I \subseteq X$ is any interval, then either $I$ is a subinterval of $L$ or $R$ or one of the $I_k$, or $I$ contains a convex copy of $\mathbb{Q}[I_k]$, and these possibilities are mutually exclusive. 

If $L = R = \emptyset$, then $X = \mathbb{Q}[I_k]$ is a shuffle. We have already showed in this case that $X$ is left-absorbing, and indeed that $AX \cong X$ for every non-empty countable linear order $A$.

Suppose $L \neq \emptyset$ but $R = \emptyset$. There are two subcases. 
\begin{itemize}
    \item[i.)] There is an index $k_0$ such that $L \cong I_{k_0}$. 
    \,\ \\
    
    Then $X \cong I_{k_0} + \mathbb{Q}[I_k]$. We write $X = (1+\mathbb{Q})[I_k]$. We claim that for a countable order $A$, we have $AX \cong X$ if and only if $A$ has a left endpoint. Fix a countable order $A$ with a left endpoint. We explicitly describe an isomorphism $F: AX \rightarrow X$. Let $(1 + \mathbb{Q}) = \bigcup_k Q_k$ be the partition of $1 + \mathbb{Q}$ according to the replacement by the orders $\{I_k\}$, so that for each fixed $k$, each $q \in Q_k$ is replaced by $I_k$. Note that the left endpoint of $1 + \mathbb{Q}$ belongs to $Q_{k_0}$. Let $A(1 + \mathbb{Q}) = \bigcup_k R_k$ be the corresponding partition of $A(1 + \mathbb{Q})$, and note that the left endpoint of $A(1 + \mathbb{Q})$ is in $R_{k_0}$. By Skolem's theorem there is an isomorphism $f: A(1 + \mathbb{Q})\rightarrow 1 + \mathbb{Q}$ such that $f[Q_k] = R_k$ for all $k$. View $AX$ as $A(1 + \mathbb{Q})[I_k]$ and $X$ as $(1+\mathbb{Q})[I_k]$, and define $F: AX \rightarrow X$ by $F(a, q, i) = (f(a, q), i)$. Then $F$ is well-defined by our choice of $f$, and readily seen to be an isomorphism of $AX$ with $X$, as desired.

    On the other hand, if $A$ has no left endpoint, then every non-empty initial segment of $AX$ contains a convex copy of $X$. It follows there is no initial segment of $AX$ that is isomorphic to $L$, and thus $AX \not\cong X$.
    \,\ \\

    \item[ii.)] There is no such $k_0$. That is, for every $k$ we have $L \not\cong I_k$. 
    \,\ \\
    
    We show that $X$ is not left-absorbing. Fix an order $A$ with at least two points $a < b$. View  the product $AX$ as a replacement in which every point is replaced by $X$, and consider the interval $I_b \subseteq AX$. Since $I_b$ is a convex copy of $X$, it contains an initial convex copy of $L$ that we denote $L_b$. Suppose that $J \supseteq L_b$ is an interval that strictly contains $L_b$. We claim that $J$ contains a convex copy of $\mathbb{Q}[I_k]$. If $J$ extends $L_b$ to the right, this is immediate, since then $J$ contains an initial segment of the copy of $\mathbb{Q}[I_k]$ adjacent to $L_b$ in $I_b$, and hence $J$ contains a convex copy of $\mathbb{Q}[I_k]$. So suppose $J$ only extends $L_b$ to the left. Either $b$ has a predecessor $b' \in A$, or it does not. If there is such a $b'$, then $I_{b'}$ immediately precedes $I_b$ in $AX$, and $J$ contains a final segment of the copy of $\mathbb{Q}[I_k]$ in $I_{b'} \cong L + \mathbb{Q}[I_k]$. But $J$ contains a convex copy of $\mathbb{Q}[I_k]$, since every final segment of $\mathbb{Q}[I_k]$ contains a convex copy of itself. And if there is no such $b'$, then $J$ in fact contains $I_c$ for infinitely many $c < b$ in $A$, and therefore many convex copies of $\mathbb{Q}[I_k]$. Thus $J$ always contains a convex copy of $\mathbb{Q}[I_k]$, as claimed. But then $L_b$ is a non-initial interval in $AX$ that contains no convex copy of $X$, and all of whose strict superintervals do contain such a copy. The only non-initial intervals in $X = L + \mathbb{Q}[I_k]$ that have this property are isomorphic to one of the $I_k$. But $L_b$ is not isomorphic to any of the $I_k$. It follows $AX \not\cong X$. 
\end{itemize}
\,\

If instead $L = \emptyset$ and $R \neq \emptyset$, then a symmetric argument gives that $X$ is left-absorbing if and only if $R \cong I_{k_0}$ for some $k_0$. In this case, $X$ absorbs $A$ if and only if $A$ is countable and has a right endpoint. 

Finally, suppose that both $L$ and $R$ are nonempty. This is the most elaborate case. To analyze it, it will be helpful to recall some basic facts about the so-called \emph{finite condensation} $\sim_{Fin}$. For more on the finite condensation, see \cite[Ch. 4]{Rosenstein}. Given a linear order $M$ and $x, y \in M$, define $x \sim_{Fin} y$ if the interval $[\{x, y\}]$ is finite. It is easily verified that $\sim_{Fin}$ is a condensation. Given $x \in M$, it need not be true that the condensation class $c_{Fin}(x)$ is finite. However, what is true is that $c_{Fin}(x)$ is either finite, or isomorphic to one of $\mathbb{N}, \mathbb{N}^*$, or $\mathbb{Z}$. Suppose that we have condensation classes $c_{Fin}(x)$ and $c_{Fin}(y)$, and $c_{Fin}(y)$ is the successor of $c_{Fin}(x)$ in $M / \sim_{Fin}$. Then if $c_{Fin}(x)$ has a right endpoint (i.e. $c_{Fin}(x)$ is finite or isomorphic to $\mathbb{N}^*$), it must be that $c_{Fin}(y)$ does not have a left endpoint (i.e. $c_{Fin}(y)$ is isomorphic to either $\mathbb{N}^*$ or $\mathbb{Z}$). Otherwise this left endpoint would be related by $\sim_{Fin}$ to the right endpoint of $c_{Fin}(x)$, contradicting that these points come from distinct condensation classes. Symmetrically, if $c_{Fin}(y)$ has a left endpoint, it must be that $c_{Fin}(x)$ does not have a right one. It follows that, in general, if a given class $c_{Fin}(x)$ has any endpoints, these endpoints are limit points in $M$.

Returning to our final case with both $L$ and $R$ nonempty, there are essentially two ways that $X$ can be left-absorbing. The first, that mirrors the previous cases, is that there are indices $k_0$ and $k_1$ such that $L$ is isomorphic to $I_{k_0}$ and $R$ is isomorphic to $I_{k_1}$. The second, more novel way is that there is an index $k$ such that $L$ is isomorphic to a final segment of $I_k$ and $R$ is isomorphic to the corresponding initial segment, that is, $R + L \cong I_k$. These conditions are not exclusive. Indeed in the second case it may be that neither, one, or both of $R$ and $L$ are isomorphic to one of the $I_k$ and we still get left absorption. Which orders $A$ are absorbed by $X$ depends on which combination of these conditions holds.

Suppose first that $R + L \not\cong I_k$ for any index $k$. There are two subcases.

\begin{itemize}
    \item[i.)] There exist indices $k_0$ and $k_1$ such that $L \cong I_{k_0}$ and $R \cong I_{k_1}$. 
    \,\ \\
    
    We claim that $AX \cong X$ if and only if $A \cong 1 + \mathbb{Q} + 1$, or $A \cong 1$. If $A \cong 1 + \mathbb{Q} + 1$, view $AX$ as $(1 + \mathbb{Q} + 1)(I_{k_0} + \mathbb{Q}[I_k] + I_{k_1})$. Condense $AX$ onto $(1 + \mathbb{Q} + 1)(1 + \mathbb{Q} + 1)$ by condensing every $I_k$ to a point. View $(1 + \mathbb{Q} + 1)(1 + \mathbb{Q} + 1)$ as $(1 + \mathbb{Q} + 1)$. It is not hard to check that for each fixed index $k$, set of points in $1 + \mathbb{Q} + 1$ that are condensed images of $I_k$ is dense in $(1 + \mathbb{Q} + 1)$. Moreover, the left endpoint of $(1 + \mathbb{Q} + 1)$ is a condensed copy of $I_{k_0}$ and right endpoint is a condensed copy of $I_{k_1}$. It follows that $AX \cong I_{k_0} + \mathbb{Q}[I_k] + I_{k_1} = X$. 

    Conversely, suppose $A \not\cong 1 + \mathbb{Q} + 1$ and $A$ has at least two points. If $A$ is uncountable, then certainly $AX \not \cong X$. If $A$ has no left endpoint, then every initial segment of $AX$ contains a convex copy of $\mathbb{Q}[I_k]$, so that $AX \not\cong X$. Symmetrically, if $A$ has no right endpoint we have $AX \not\cong X$. Finally, suppose that $A$ is not dense. Fix $a, b \in A$ with $b$ the successor of $a$. View $AX$ as a replacement. The convex copies $I_a$ and $I_b$ of $X$ are adjacent in $AX$. Since $I_a$ contains a final segment $R_a$ that isomorphic to $R$ and $I_b$ contains an initial segment $L_b$ that is isomorphic to $L$, the interval $R_a + L_b$ is isomorphic to $R + L$. Since neither $R$ nor $L$ contains a convex copy of $\mathbb{Q}[I_k]$, neither does $R_a + L_b$, by the self-similarity of $\mathbb{Q}[I_k]$. It is not hard to see however that every interval that strictly contains $R_a + L_b$ contains a copy of $\mathbb{Q}[I_k]$. The only intervals in $X$ that are neither initial nor final segments, and that are maximal with respect to not containing a convex copy of $X$, are the intervals $I_k$. But by hypothesis $R_a + L_b$ is not isomorphic to $I_k$ for any $k$. Since $R_a + L_b$ is neither initial nor final in $AX$, it follows $AX \not\cong X$ as claimed. 
    \,\ \\

    \item[ii.)] At least one of $R, L$ is not isomorphic to any $I_k$. Suppose it is $L$, without loss of generality. 
    \,\ \\
    
    We show that $X$ is not left-absorbing. Suppose that $A$ is a countable order with at least two points. A similar argument to the one given in (i.) above shows that if $A$ is missing either a left or right endpoint, or if $A$ is not dense, then $AX \not\cong X$. It remains only to check that $AX \not \cong X$ when $A = 1 + \mathbb{Q} + 1$. Fix any $q \in \mathbb{Q}$ and consider the convex copy $I_q$ of $X$ in $AX$. This copy contains an initial segment $L_q$ that is isomorphic to $L$. Any interval extending $L_q$ to the right intersects the copy of $\mathbb{Q}[I_k]$ in $I_q \cong L_q + \mathbb{Q}[I_k] + R_q$ and hence contains a convex copy of $\mathbb{Q}[I_k]$. And since $q$ is a limit point (from the left) of $A = 1 + \mathbb{Q} + 1$, any interval properly extending $L_q$ to the left contains infinitely many of the $I_r$ for $r < q$, and thus infinitely many convex copies of $\mathbb{Q}[I_k]$. Thus $AX$ contains an interval that is neither final nor initial, that is not isomorphic to any $I_k$, and is maximal with respect to not containing a convex copy of $\mathbb{Q}[I_k]$. There is no such interval in $X$, so that $AX \not\cong X$.
\end{itemize}

Now suppose there is $k = k_0$ such that $R + L = I_{k_0}$. We work through four subcases.

\begin{itemize}
    \item[i.)] Neither $R$ nor $L$ is isomorphic to any of the orders $I_k$. 
    \,\ \\
    
    We claim that that we have $AX \cong X$ if and only if $A$ is countable and has a left and right endpoint, and every point in $A$ that is not the left endpoint has a successor, and every point that is not the right endpoint has a predecessor. 
    
    Suppose first that $A$ has these properties. It will be useful to get a more specific description of $A$. Consider the finite condensation $A / \sim_{Fin}$. There are two possibilities. The first is that $A / \sim_{Fin}$ consists of a single condensation class. In this case, it must be that this condensation class contains both of the endpoints of $A$, and hence must be finite, since the only possible $\sim_{Fin}$-classes with two endpoints are the finite ones. Thus $A$ is isomorphic to a finite order $n$. 

    The second possibility is that $A / \sim_{Fin}$ has cardinality at least two. Then the hypotheses on $A$, along with the facts about the possible forms of $\sim_{Fin}$ condensation classes recalled above, imply that the condensation class containing the left endpoint of $A$ is isomorphic to $\mathbb{N}$, the class containing the right endpoint is isomorphic to $\mathbb{N}^*$, and every other condensation class is isomorphic to $\mathbb{Z}$. That is, we have $A \cong \mathbb{N} + B\mathbb{Z} + \mathbb{N}^*$ for some countable order $B$. It is possible that $B$ is empty, in which case $A \cong \mathbb{N} + \mathbb{N}^*$.

    In the first case, when $A \cong n$, we have
    \[
    \begin{array}{l l l}
    AX & \cong & nX \\
    & \cong & X + X + \ldots + X \\
    & = & L + \mathbb{Q}[I_k] + R + L + \mathbb{Q}[I_k] + R + \ldots + L + \mathbb{Q}[I_k] + R \\
    & \cong & L + \mathbb{Q}[I_k] + I_{k_0} + \mathbb{Q}[I_k] + I_{k_0} + \ldots + I_{k_0} + \mathbb{Q}[I_k] + R.
    \end{array}
    \]

    If we condense the intervals $I_k$ in this order (leaving $L$ and $R$ uncondensed), we get a sum of the form $L + \mathbb{Q} + 1 + \mathbb{Q} + 1 + \ldots + 1 + \mathbb{Q} + R$, which is isomorphic to  $L + \mathbb{Q} + R$. For each fixed $k$, the points in the central copy of $\mathbb{Q}$ that are condensed images of $I_k$ are dense in $\mathbb{Q}$. It follows that the original uncondensed order $AX \cong nX$ is isomorphic to $L + \mathbb{Q}[I_k] + R \cong X$, as desired.

    In the second case we have $AX \cong (\mathbb{N} + B\mathbb{Z} + \mathbb{N}^*)X \cong \mathbb{N}X + B\mathbb{Z}X + \mathbb{N}^*X$. The lefthand term can be expanded as follows:
    \[
    \begin{array}{l l l}
    \mathbb{N}X & \cong & X + X + \ldots \\
    & = & L + \mathbb{Q}[I_k] + R + L + \mathbb{Q}[I_k] + R + \ldots \\
    & \cong & L + \mathbb{Q}[I_k] + I_{k_0} + \mathbb{Q}[I_k] + I_{k_0} + \ldots 
    \end{array}
    \]
    If we condense the intervals $I_k$ in this sum, we get a sum of the form $L + \mathbb{Q} + 1 + \mathbb{Q} + 1 + \ldots$ which is isomorphic to $L + \mathbb{Q}$. Since for each $k$, the points in $\mathbb{Q}$ that are condensed images of $I_k$ are dense in $\mathbb{Q}$, we have that $\mathbb{N}X$ is isomorphic to $L + \mathbb{Q}[I_k]$.

    By similar arguments we have that $\mathbb{N}^*X \cong \mathbb{Q}[I_k] + R$ and $\mathbb{Z}X \cong \mathbb{Q}[I_k]$. Thus, in the case when $B \neq \emptyset$, we have
    \[
    \begin{array}{l l l}
    AX & \cong & \mathbb{N}X + B\mathbb{Z}X + \mathbb{N}^*X \\
    & \cong & L + \mathbb{Q}[I_k] + B\mathbb{Q}[I_k] + \mathbb{Q}[I_k] + R \\
    & \cong & L + \mathbb{Q}[I_k] + \mathbb{Q}[I_k] + \mathbb{Q}[I_k] + R \\
    & \cong & L + \mathbb{Q}[I_k] + R \\
    & \cong & X 
    \end{array}
    \]
    as claimed. The case when $B = \emptyset$ is similar. 

    Conversely, suppose that $A$ has at least two points and does not satisfy the hypotheses above. We show that $AX \not\cong X$. If $A$ is missing either a left or right endpoint, then $AX$ will either have no initial segment isomorphic to $L$, or no final segment isomorphic to $R$, and hence $AX \not \cong X$. So suppose $A$ has both endpoints. Then either there is some $a \in A$, different from the left endpoint, that has no successor, or some $a' \in A$, different from the right endpoint, that has no predecessor. Given such an $a$, consider the corresponding convex copy $I_a$ of $X$ in $AX$. Then $I_a$ has an initial segment $L_a$ that is isomorphic to $L$. Note that $L_a$ is not initial or final in $X$. Since $a$ has no predecessor in $X$, every interval strictly containing $L_a$ contains a convex copy of $X$. But $L_a$ contains no convex copy of $X$. The only non-initial and non-final intervals with this property in $X$ are the $I_k$, and $L_a$ is not isomorphic to any such interval by hypothesis. Thus $AX \not\cong X$. The argument is symmetric given such an $a'$. Thus $AX \not\cong X$. 
    \,\ \\

    \item[ii.)] There is an index $k_1$ such that $L \cong I_{k_1}$, but there is no index $k$ such that $R \cong I_k$. 
    \,\ \\
    
    We show that $AX \cong X$ if and only if $A$ is countable, has a left and right endpoint, and every point other than the right endpoint has a successor. Or equivalently, by way of the finite condensation, we have $AX \cong X$ if and only if $A$ is finite, or for some countable order $B$ we have $A \cong \mathbb{N} + B(I_b) + M$, where for every $b \in B$ either $I_b \cong \mathbb{N}$ or $I_b \cong \mathbb{Z}$, and either $M \cong n$ for some finite $n$, or $M \cong \mathbb{N}^*$. 

    If $A$ is finite, then the same argument given in the previous case shows that $AX \cong X$. So suppose $A \cong \mathbb{N} + B(I_b) + M$. Then $AX \cong \mathbb{N}X + B(I_b)X + MX$. As before, we have that $\mathbb{N}X \cong L + \mathbb{Q}[I_k]$. If $M \cong \mathbb{N}^*$, then $MX \cong \mathbb{Q}[I_k] + R$ as above, and if $M \cong n$ for some finite order $n$, then $MX \cong L + \mathbb{Q}[I_k] + R$. For the middle term, distributing the $X$ over the replacement we get $B(I_b)X \cong B(I_bX)$. Each term $I_bX$ is isomorphic to either $L + \mathbb{Q}[I_k]$ or $\mathbb{Q}[I_k]$, depending on whether $I_b \cong \mathbb{N}$ or $I_b \cong \mathbb{Z}$. Putting this together, we have $AX \cong L + \mathbb{Q}[I_k] + B(J_b) + L' + \mathbb{Q}[I_k] + R$, where for every $b \in B$, the replacing order $J_b$ is either isomorphic to $L + \mathbb{Q}[I_k]$ or $\mathbb{Q}[I_k]$, and $L'$ is either empty or equal to $L$. If we condense the intervals $I_k$ appearing in this sum (including the segments $L \cong I_{k_1}$), we get an order of the form $1 + \mathbb{Q} + B(K_b) + p + \mathbb{Q} + R$, where for each $b \in B$, $K_b$ is either isomorphic to $1 + \mathbb{Q}$ or $\mathbb{Q}$, and $p$ is either empty or $1$. Any replacement in which the replacing orders are either $\mathbb{Q}$ or $1+\mathbb{Q}$ is dense as a linear order and has no right endpoint. It follows that $B(K_b)$ is isomorphic to one of $1 + \mathbb{Q}$ and $\mathbb{Q}$. Thus our condensation of $AX$ is isomorphic to $(1 + \mathbb{Q}) + (p_0 + \mathbb{Q}) + (p_1 + \mathbb{Q} + R)$, where $p_0$ is either empty or $1$, and likewise for $p_1$. No matter the values of $p_0$ and $p_1$, this order is isomorphic to $1 + \mathbb{Q} + R$. Here, the leading $1$ here is a condensed image of $I_{k_1} \cong L$, and it is not hard to see that for each fixed $k$, the points in $\mathbb{Q}$ that are condensed images of $I_k$ are dense in $\mathbb{Q}$. It follows $AX \cong L + \mathbb{Q}[I_k] + R = X$, as desired. We implicitly assumed $B \neq \emptyset$. The case when $B = \emptyset$ is similar.

    Now suppose that $A$ does not satisfy the hypotheses above. If $A$ is missing a left or right endpoint, then $AX \not\cong X$, since either every initial segment of $AX$ contains a convex copy of $X$, or every final segment does. If $A$ has both endpoints, there must be a point $a \in A$, distinct from its right endpoint, that has no successor. But then the convex copy $I_a$ of $X$ in the product $AX$ has a final segment $R_a$ that is not isomorphic to any $I_k$ and is not initial or final in $A$, and is maximal with respect to not containing a convex copy of $X$. Since $X$ has no such interval, we have $AX \not\cong X$.
    \,\ \\

    \item[iii.)] There is an index $k_2$ such that $R \cong I_{k_2}$, but there is no index $k$ such that $L \cong I_k$.
    \,\ \\
    
    By a symmetric argument to the previous case, we have that $AX \cong X$ if and only if $A$ is countable, has both a left and right endpoint, and every point that is not the left endpoint has a predecessor.
    \,\ \\

    \item[iv.)] There are indices $k_1$ and $k_2$ such that $L \cong I_{k_1}$ and $R \cong I_{k_2}$. 
    \,\ \\
    
    We claim that $AX \cong X$ if and only if $A$ is countable and has both a left and right endpoint. If $A$ is missing either endpoint, then as before we have $AX \not\cong X$. So suppose that $A$ is countable and has an endpoint on each side, that we label $l$ and $r$. Consider the finite condensation $A / \sim_{Fin}$ of $A$. Let $B = A / \sim_{Fin}$, so that $A \cong B(I_b)$, where each $\sim_{Fin}$ condensation class $I_b$ is isomorphic to some finite order $n$, or $\mathbb{N}$, or $\mathbb{N}^*$, or $\mathbb{Z}$. Since $A$ has a left endpoint, $B$ must have a left endpoint also, corresponding to the condensation class that contains the left endpoint of $A$. We denote this class as $I_l$. Correspondingly there is a rightmost condensation class in $B(I_b)$ that we label $I_r$. If $I_l = I_r$ then there is a single, finite condensation class, and we have $AX \cong X$ by the same proof as before. So we assume that $I_l \neq I_r$. By the properties of the finite condensation mentioned above, if there are points $b, b' \in B$ with $b'$ the successor of $b$, then if $I_b$ has a right endpoint, it must be that $I_{b'}$ has no left endpoint, and likewise if $I_{b'}$ has a left endpoint, then $I_b$ has no right one. 
    
    We have that $AX \cong B(I_b)X \cong B(I_bX)$. For a fixed $b \in B$, the order $I_bX$ is isomorphic to either $L + \mathbb{Q}[I_k] + R$, $L + \mathbb{Q}[I_k]$, $\mathbb{Q}[I_k] + R$, or $\mathbb{Q}[I_k]$, depending on whether $I_b$ is isomorphic to a finite order, or $\mathbb{N}$, or $\mathbb{N}^*$, or $\mathbb{Z}$, respectively. If we condense each of the intervals $I_k$ in the terms $I_bX$, including the initial copies of $L \cong I_{k_1}$ and final copies of $R \cong I_{k_2}$ in the terms that include them, we obtain from $B(I_bX)$ a condensed order $B(K_b)$, where each $K_b$ is either $1 + \mathbb{Q} + 1$, or $1 + \mathbb{Q}$, or $\mathbb{Q} + 1$, or $\mathbb{Q}$, respectively. Note that the condensed term $K_b$ has a left endpoint if and only if the original finite condensation class $I_b$ has a left endpoint, and likewise for right endpoints. 

    We claim that the condensed order $B(K_b)$ is isomorphic to $1 + \mathbb{Q} + 1$. Since $B(I_bX)$ has a leftmost term $I_lX$, and $I_l$ is the finite condensation class of $A$ that includes its left endpoint, it must be that $I_l$ is finite or isomorphic to $\mathbb{N}$. Thus $I_lX$ is either $L + \mathbb{Q}[I_k] + R$ or $L + \mathbb{Q}[I_k]$, so that its condensation $K_l$ is either $1 + \mathbb{Q} + 1$ or $1 + \mathbb{Q}$. In any case, $K_l$ has a left endpoint, which is also the left endpoint of $B(K_b)$. Symmetrically, $B(K_b)$ has a right endpoint. It remains only to check that $B(K_b)$ is dense. Fix $x < y$ in $B(K_b)$. Since each individual term $K_b$ is dense, the only possible way $y$ could be a successor of $x$ is if $x$ is the right endpoint of some $K_b$ and $y$ is the left endpoint of some $K_c$, where $c$ is the successor of $b$ in $B$. But by our observation in the previous paragraph, this would mean that in $B = A / \sim_{Fin}$ the finite condensation class $I_b$ has a right endpoint and is succeeded by the condensation class $I_c$ with a left endpoint, which is impossible. 

    Thus $B(K_b)$ is dense, and therefore isomorphic to $1 + \mathbb{Q} + 1$. We view it as a copy of $1 + \mathbb{Q} + 1$. It is not hard to verify that for each fixed $k$, the set of points in this copy of $1 + \mathbb{Q} + 1$ that are condensed images of $I_k$ is dense. Since the first point is a condensed image of $L$ and the last a condensed image of $R$, we have that the uncondensed order $AX \cong B(I_bX)$ is isomorphic to $L + \mathbb{Q}[I_k] + R \cong X$, as desired. 
\end{itemize}

This finishes our classification and proves the following theorem.

\theoremstyle{definition}
\newtheorem{thm2}[thm1]{Theorem}
\begin{thm2}\label{thm2} \,\ 
Suppose that $X$ is a countable linear order. Then $X$ is left-absorbing if and only if there is a countable, minimal collection $\{I_k\}$ of non-empty countable linear orders, and two countable orders $L$ and $R$, such that $X = L + \mathbb{Q}[I_k] + R$ and exactly one of the following holds:
\begin{itemize}
    \item[1.] $L = R = \emptyset$,
    \item[2.] There is an index $k$ such that $L \cong I_k$, and $R = \emptyset$,
    \item[3.] There is an index $k$ such that $R \cong I_k$, and $L = \emptyset$,
    \item[4.] There are indices $k_1$ and $k_2$ such that $L \cong I_{k_1}$ and $R \cong I_{k_2}$, but there is no index $k$ such that $R + L \cong I_k$,
    \item[5.] There is an index $k_0$ such that $R + L \cong I_{k_0}$, but no index $k$ such that $L \cong I_k$ and no index $k$ such that $R \cong I_k$,
    \item[6.] There is an index $k_0$ such that $R + L \cong I_{k_0}$ and an index $k_1$ such that $L \cong I_{k_1}$, but no index $k$ such that $R \cong I_k$,
    \item[7.] There is an index $k_0$ such that $R + L \cong I_{k_0}$ and an index $k_1$ such that $R \cong I_{k_1}$, but no index $k$ such that $L \cong I_k$,
    \item[8.] There are indices $k_0, k_1, k_2$ such that $R + L \cong I_{k_0}$, $L \cong I_{k_1}$, and $R \cong I_{k_2}$.
\end{itemize}

Suppose that $X$ is left-absorbing, and $A$ is a countable order. Whether $AX \cong X$ is determined by which of the conditions above that $X$ satisfies, as follows.

\begin{itemize}
    \item[i.] if (1), then $AX \cong X$ for every countable order $A$,
    \item[ii.] if (2), then $AX \cong X$ if and only if $A$ has a left endpoint,
    \item[iii.] if (3), then $AX \cong X$ if and only if $A$ has a right endpoint,
    \item[iv.] if (4), then $AX \cong X$ if and only if $A \cong 1 + \mathbb{Q} + 1$ or $A \cong 1$,
    \item[v.] if (5), then $AX \cong X$ if and only if $A$ has a left and right endpoint, and every point in $A$ other than the right endpoint has a successor, and every point in $A$ other than the left endpoint has a predecessor,
    \item[vi.] if (6), then $AX \cong X$ if and only if $A$ has a left and right endpoint, and every point in $A$ other than the right endpoint has a successor,
    \item[vii.] if (7), then $AX \cong X$ if and only if $A$ has a left and right endpoint, and every point in $A$ other than the left endpoint has a predecessor,
    \item[viii.] if (8), then $AX \cong X$ if and only if $A$ has a left and right endpoint. 
\end{itemize} \qed
\end{thm2}

We conclude with a corollary that answers the following question left open in \cite{Ervin}: which countable linear orders $X$ with both a left and right endpoint are isomorphic to their lexicographic squares?

\theoremstyle{definition}
\newtheorem{cor3}[thm1]{Corollary}
\begin{cor3}\label{cor3} \,\ 
Suppose that $X$ is a countable linear order with both a left and right endpoint. Then $X \cong X^2$ if and only if there is a countable, minimal collection $\{I_k\}$ of countable linear orders, a nonempty countable order $L$ with a left endpoint, and a non-empty countable order $R$ with a right endpoint such that $X \cong L + \mathbb{Q}[I_k] + R$ and exactly one of the following holds:
\begin{itemize}
    \item[1.] $R \cong L \cong 1$, and $I_k \cong 1$ for every index $k$, that is, $X \cong 1 + \mathbb{Q} + 1$,
    \item[2.] There is an index $k_0$ such that $R + L \cong I_{k_0}$, but no index $k$ such that $L \cong I_k$ and no index $k$ such that $R \cong I_k$; furthermore, for every index $k$, we have that every point in $I_k$ has both a successor and predecessor,
    \item[3.] There is an index $k_0$ such that $R + L \cong I_{k_0}$ and an index $k_1$ such that $L \cong I_{k_1}$, but no index $k$ such that $R \cong I_k$; furthermore, for every index $k$, we have that every point in $I_k$ has a successor,
    \item[4.] There is an index $k_0$ such that $R + L \cong I_{k_0}$ and an index $k_1$ such that $R \cong I_{k_1}$, but no index $k$ such that $L \cong I_k$; furthermore, for every index $k$, we have that every point in $I_k$ has a predecessor,
    \item[5.] There are indices $k_0, k_1, k_2$ such that $R + L \cong I_{k_0}$, $L \cong I_{k_1}$, and $R \cong I_{k_2}$.
\end{itemize} \qed
\end{cor3}

\section{Open problems}

What can we say about uncountable self-similar orders and uncountable left-absorbing orders? If $X$ is a self-similar order of any cardinality, then as in the proof of Theorem \ref{thm1} we can define a relation $\sim$ on $X$ by the rule $x \sim y$ if the interval $[\{x, y\}]$ does not contain a convex copy of $X$. By the same argument given in the proof, $\sim$ is a condensation of $X$ that organizes $X$ into convex equivalence classes $c(x)$ that are maximal with respect to not containing a convex copy of $X$, and the condensed order $X / \sim$ is dense. If $X / \sim$ is countable, it is isomorphic to one of $\mathbb{Q}$, $1 + \mathbb{Q}$, $\mathbb{Q} + 1$, or $1 + \mathbb{Q} + 1$, and it follows that $X$ is isomorphic to an order of the form $L + \mathbb{Q}[I_k] + R$ in the same sense as before, except now the orders $I_k$ appearing in the shuffle $\mathbb{Q}[I_k]$ may be uncountable. Theorems 1 and 2 apply to such $X$, otherwise verbatim.

If $X / \sim$ is uncountable, the situation is much more open. Even for self-similar orders of size $\aleph_1$, any classification will likely depend on the particular model of set theory in which we work. If the continuum hypothesis holds, then the class of orders of size $\aleph_1$ contains as a subclass all suborders of $\mathbb{R}$ of the same cardinality as $\mathbb{R}$, a class that is known to be extremely wild. Perhaps a more tractable problem than classifying the self-similar orders of size $\aleph_1$ in general is to classify such orders under the assumption of a strong set-theoretic forcing axiom like the Proper Forcing Axiom (PFA). PFA implies that the continuum is $\aleph_2$. Moreover, there is a well-developed structure theory for the class of linear orders of size $\aleph_1$ under PFA; see \cite[Section 3]{Moore}. This structure theory may help in understanding the self-similar orders and left-absorbing orders of size $\aleph_1$. 
\,\ \\

\underline{Problem 1}: Assume PFA. Classify the self-similar and left-absorbing orders of cardinality $\aleph_1$.
\,\ \\

Another approach is to study abstractly the classes of orders that are absorbed by some fixed order. For a linear order $X$, define the \emph{absorption spectrum} of $X$ to be the class of order types $\mathscr{A}_X = \{A: AX \cong X\}$. Observe that $1 \in \mathscr{A}_X$ for every $X$, where $1$ denotes the order type of a singleton. An order $X$ is left-absorbing if $\mathscr{A}_X$ contains an order type other than $1$. 

Theorem \ref{thm2} can be viewed as a classification of the classes of order types $\mathscr{A}$ for which there is a countable order $X$ such that $\mathscr{A} = \mathscr{A}_X$. For example, if $\mathscr{A}$ consists of the countable order types with a left endpoint, then $\mathscr{A} = \mathscr{A}_X$ for any countable order $X$ of the form $X = I_{k_0} + \mathbb{Q}[I_k]$. 

We isolate several abstract properties of absorption spectra. 

\theoremstyle{definition}
\newtheorem{prop4}[thm1]{Proposition}
\begin{prop4}\label{prop4} \,\ 
Suppose that $X$ is a linear order and $\mathscr{A}_X$ is its absorption spectrum. Then:
\begin{itemize}
    \item[1.] $1 \in \mathscr{A}_X$,
    \item[2.] If $A \in \mathscr{A}_X$, and $I_a \in \mathscr{A}_X$ for every $a \in A$, then $A(I_a) \in \mathscr{A}_X$,
    \item[3.] For all order types $A$ and $B$, we have $A + 1 + B \in \mathscr{A}_X$ if and only if $A + 1 \in \mathscr{A}_X$ and $1 + B \in \mathscr{A}_X$.
\end{itemize}
\end{prop4} 

\begin{proof}
We already observed $(1.)$, and $(2.)$ follows from the fact that products distribute over replacements on the right, so that $A(I_a)X \cong A(I_aX) \cong AX \cong X$. 

For $(3.)$, we will need the following fact, due to Lindenbaum: if $A$ and $B$ are linear orders such that $A$ is isomorphic to an initial segment of $B$ and $B$ is isomorphic to a final segment of $A$, then $A \cong B$. Suppose first that $A + 1 + B \in \mathscr{A}_X$. We prove that $A + 1 \in \mathscr{A}_X$, i.e. that $(A + 1)X \cong X$. Observe that $(A + 1)X$ is isomorphic to an initial segment of $(A + 1 + B)X \cong (A+1)X + BX$. Since $A + 1 + B \in \mathscr{A}_X$, it follows $(A + 1)X$ is isomorphic to an initial segment of $X$. On the other hand, $X$ is isomorphic to a final segment of $(A + 1)X \cong AX + X$. By Lindenbaum's theorem, we have $(A+1)X \cong X$, as desired. The proof that $(1 + B)X \cong X$ is symmetric.

Conversely, suppose that $A + 1$ and $1 + B$ belong to $\mathscr{A}_X$. Then $AX + X \cong X + BX \cong X$. Observe that $(A + 1 + B)X \cong AX + X + BX$. Since $X + BX \cong X$, we have $AX + (X + BX) \cong AX + X \cong X$, giving $A + 1 + B \in \mathscr{A}_X$, as desired. 
\end{proof}
\,\ 

Condition (3.) in the proposition may seem peculiar, but in conjunction with condition (1.) it implies the following perhaps more intuitive closure property of $\mathscr{A}_X$: if $A \in \mathscr{A}_X$ and $[a, a']$ is a closed interval in $A$, then $[a, a']$ (viewed as an order type) belongs to $\mathscr{A}_X$.

Suppose that $\mathscr{A}$ is a class of order types satisfying the conditions $(1.)$, $(2.)$, and $(3.)$ from Proposition \ref{prop4}. Is $\mathscr{A}$ the absorption spectrum for some order $X$? Not necessarily. For example, if $\mathscr{A}$ consists of all order types of cardinality at most $\aleph_1$, then $\mathscr{A}$ satisfies $(1.) - (3.)$. But for any given order $X$, we have $\mathbb{N} X \not\cong \omega_1 X$, since $\mathbb{N} X$ and $\omega_1 X$ have distinct cofinalities. Thus it cannot be that $X \cong \mathbb{N} X \cong \omega_1 X$. Since $\mathbb{N}$ and  $\omega_1$ belong to $\mathscr{A}$, it follows $\mathscr{A} \neq \mathscr{A}_X$. 
\,\ \\

\underline{Problem 2}. Are there conditions extending those from Proposition \ref{prop4} such that a class of order types $\mathscr{A}$ satisfies the conditions if and only if $\mathscr{A} = \mathscr{A}_X$ for some order $X$?
\,\ \\

\underline{Problem 3}. Fix a left-absorbing order $X$. What can be said about the orders $Y$ such that $\mathscr{A}_Y = \mathscr{A}_X$?
\,\ \\

\end{document}